\documentclass[a4paper]{article}
\usepackage[english]{babel}
\usepackage[utf8]{inputenc}
\usepackage{amsmath}
\usepackage{amssymb}
\usepackage{amsthm}
\usepackage{amscd}
\usepackage{amsfonts}
\usepackage{newlfont}
\usepackage{enumerate}
\usepackage{graphicx}
\usepackage[misc]{ifsym}
\usepackage{titlefoot}
\usepackage[colorlinks=true]{hyperref}
\hypersetup{urlcolor=blue, citecolor=red}
\allowdisplaybreaks

\DeclareMathOperator{\sgn}{sgn}

\newtheorem{thm}{Theorem}[section]

\newtheorem{prop}[thm]{Proposition}
\newtheorem{definition}[thm]{Definition}
\newtheorem{rem}[thm]{Remark}
\newtheorem{corol}[thm]{Corollary}

\newcommand{\KeyWords}[1]
{
  {\small	
  \textbf{Keywords:} #1}
}
\newcommand{\msc}[1]
{
  {\small	
  \textbf{2020 Mathematics Subject Classification:} #1}
}

\title{Differential inclusions and quasi-Lyapunov functions}

\author{Martin Ivanov \and Mikhail Krastanov \and Nadezhda Ribarska}
\date{}

\begin{document}

\maketitle

\begin{abstract}
A sufficient condition for existence of a solution of a differential inclusion with a uniformly bounded right-hand side that has nonempty closed (possibly nonconvex) values is obtained. A new way to deal with the so-called ``bad points'' is proposed. An Olech-type result is obtained as a corollary. An example, which originates from the Fuller problem from optimal control theory, is given to demonstrate the applicability of the main result.
\end{abstract} \vspace{5mm}

\vspace{5mm}

\msc{34A36, 34A60}

\unmarkedfntext{
\hspace*{-2em}
Martin Ivanov\\
Faculty of Mathematics and Informatics, Sofia University “St. Kliment Ohridski”, 5 James Bourchier Blvd., 1164 Sofia, Bulgaria, e-mail: \href{mailto:mdobrinov@fmi.uni-sofia.bg}{mdobrinov@fmi.uni-sofia.bg}
\vspace{1.5mm}

\hspace*{-1em}Mikhail Krastanov\\
Faculty of Mathematics and Informatics, Sofia University “St. Kliment Ohridski”, 5 James Bourchier Blvd., 1164 Sofia, Bulgaria and Institute of Mathematics and Informatics, Bulgarian Academy of Sciences, Acad. G. Bonchev str., bl. 8, 1113 Sofia, Bulgaria, e-mail: \href{mailto:krastanov@fmi.uni-sofia.bg}{krastanov@fmi.uni-sofia.bg}
\vspace{1.5mm}

\hspace*{-1em}Nadezhda Ribarska\\
Faculty of Mathematics and Informatics, Sofia University “St. Kliment Ohridski”, 5 James Bourchier Blvd., 1164 Sofia, Bulgaria and Institute of Mathematics and Informatics, Bulgarian Academy of Sciences, Acad. G. Bonchev str., bl. 8, 1113 Sofia, Bulgaria, e-mail: \href{mailto:ribarska@fmi.uni-sofia.bg}{ribarska@fmi.uni-sofia.bg}\\}

\section{Introduction}

We study the existence of local solutions of differential inclusions of the form $$\dot{x} \in F(x),\, x(0)=x_0,$$ where $x_0 \in \mathbb{R}^n$ and $F : \mathbb{R}^n \rightrightarrows \mathbb{R}^n$ is a multi-valued mapping with nonempty values, i.e. the existence of $T > 0$ and an absolutely continuous function $x: [0,T] \to \mathbb{R}^n$ such that $x(0)=x_0$ and $\dot{x}(t) \in F(x(t))$ for almost every $t \in [0,T]$.

It is well known that there is a solution for the case of upper semi-continuous right-hand side with nonempty convex compact values. The question whether solutions exist for the case of continuous right-hand side with nonconvex
values was solved by Filippov in 1971 (cf. \cite{F}). This was generalized to
the case of lower semi-continuous right-hand side by Bressan (1980, cf. \cite{B}) and by Lojasiewicz (1980, cf. \cite{L0}) (more detailed historical remarks can be found in \cite{D}). A longstanding open problem is the existence of a solution of a differential inclusion with upper semi-continuous right-hand side with nonconvex values. As can be seen from the well-known Filippov's examples (cf. \cite{D}), such a solution does not always exist unless some additional assumptions are made.

Existence results for differential inclusions with mixed-type conditions on the right-hand side were proved by Olech (1975, cf. \cite{O}) and Lojasiewicz (1985, cf. \cite{L}) (see also Corollary 6.4 from \cite{D} and the remarks after it). These results unify the upper semi-continuous and lower semi-continuous cases. For example, in Lojasiewicz's result, lower semi-continuity on an open set and convexity and upper semi-continuity on its complement are assumed. In \cite{Bivas} Bivas generalized all known such theorems in the finite-dimensional case. In the present paper, we propose a theorem of this type, but instead of convexity, we assume that the field ``moves away'' from the set of points, where the right-hand side lacks lower semi-continuity.  

We make use of the formalism proposed in \cite{KR}, which in a way is a development of the patchy vector fields of Ancona and Bressan (cf. \cite{AB}). This approach consists in constructing a ``suitable'' relatively open partition of the domain of the multi-valued mapping and approximating it on each element of the partition with an upper semi-continuous mapping which has convex compact values. This approximation mapping should be ``not far'' from the original one, and with respect to it the element of the partition should be weakly forward invariant (see Definition \ref{invapr} below). This construction works well in the case where the relatively open partition is locally finite. In such a case, a uniform limit of $\varepsilon$-solutions is an $\varepsilon$-solution of the differential inclusion. For the case where locally finite partitions of the above kind cannot be constructed, a more complicated structure (consisting of countably many relatively open partitions, each refining the previous one, see Definition \ref{invpapr} below) is developed in \cite{KR}.

In the present paper, we use the techniques from \cite{KR} to prove a sufficient condition for existence of a solution of a differential inclusion with uniformly bounded right-hand side that has nonempty closed values. We propose a different way of dealing with the ``bad points''. In contrast to the assumption of convex values of the mapping on the set of ``bad points'' in \cite{KR}, we assume the existence of a function that keeps the $\varepsilon$-solutions of the inclusion starting from ``good points'' away from the set of ``bad points''. In this way, we prove that it is possible to construct an $\varepsilon$-solution starting from a point on the boundary of the set of ``bad points''.

The function in question (called quasi-Lyapunov function) is assumed to be piecewise smooth. The reason is the fact that this research began as a search for conditions guaranteeing the existence of solutions of inclusions whose right-hand side belongs to some o-minimal structure (for example, subanalytic sets). In such a situation, piecewise smoothness is a natural assumption. It is possible to look for non-smooth variants.

We show that the main result is applicable to an inclusion that originates from the optimal feedback for the Fuller problem from optimal control theory. This problem is one of the first examples of chattering control -- the optimal control switches infinitely many times on a finite time interval (cf. \cite{ZB}). We construct a quasi-Lyapunov function in this example which is semi-algebraic.

The structure of the paper is as follows. In the second section, the necessary preliminary definitions and results are stated. The third section contains the main result of the paper. A definition of a quasi-Lyapunov function is given. In the presence of such a function, we prove that there exists an $\varepsilon$-solution of the inclusion starting from a ``bad point''. Under an additional assumption that involves successive refinements of the partition, the limit as $\varepsilon \to 0$ is taken and a solution of the inclusion is obtained. As a consequence, we obtain a result of Olech-type. The fourth section contains an application of the main result. The fifth section contains some concluding remarks.

\section{Preliminaries}

The multi-valued mappings under consideration are defined on a nonempty subset $D$ of $\mathbb{R}^n$ and the solutions of the inclusions have to stay in $D$ in order to be well defined. Moreover, one of the key ideas of the presented approach is the weak forward invariance of the elements of the partitions with respect to the corresponding approximating mappings, so we recall the following notions related to invariance.

\begin{definition}[\cite{KR}]
    Let $D \subseteq \mathbb{R}^n$ and $\varphi:[0,T) \to \mathbb{R}^n$ be an absolutely continuous function. It is said that the set $D$ is forward invariant with respect to $\varphi$ if for each $t \in [0,T)$ with $\varphi(t) \in D$ there exists $\delta>0$ such that $\varphi(\tau) \in D$ for all $\tau \in [t,t+\delta)$.
\end{definition}

\begin{definition}[\cite{CLSW}]
    Let $S \subseteq D \subseteq \mathbb{R}^n$ and $F: D \rightrightarrows \mathbb{R}^n$ be a multi-valued mapping. It is said that $S$ is weakly forward invariant with respect to $F$ if for every $x_0 \in S$ there exist $\delta > 0$ and a trajectory $\varphi$ of $\dot{x} \in F(x)$ defined on $[0,\delta)$ such that $\varphi(0)=x_0$ and $\varphi(t) \in S$ for all $t \in [0,\delta)$.
\end{definition}

There are different conditions that guarantee weak forward invariance (cf. Chapter 4.2 from \cite{CLSW}). The one we will use involves the existence of tangent velocities at each point of the set that belong to the following tangent cone:

\begin{definition}[\cite{AC}, \cite{CLSW}]
    Let $S \subseteq \mathbb{R}^n$ and $x \in S$. The Bouligand tangent cone to $S$ at $x$ is defined as: $$T^B_S(x) = \left\{\lim_{n \to \infty} \frac{x_n-x}{t_n} : \{x_n\}_{n=1}^{\infty} \subseteq S, x_n \xrightarrow{} x, t_n \searrow 0\right\}.$$
\end{definition}

The following theorem will guarantee the weak forward invariance of the elements of the partitions with respect to the approximating mappings.

\begin{thm}[\cite{AC}]\label{localsol}
    Let $K$ be a locally closed subset of $\mathbb{R}^n$. Let $F: K \rightrightarrows \mathbb{R}^n$ be an upper semi-continuous multi-valued mapping with nonempty convex compact values. If $F(x) \cap T^B_K(x) \ne \emptyset$ for all $x \in K$, then $K$ is weakly forward invariant with respect to $F$.
\end{thm}

In the following definitions we make precise what is meant by a ``suitable'' partition and by a refinement of it.

\begin{definition}[\cite{NR}]
    Let $X$ be a topological space and $$\mathcal{U} = \{U_{\alpha} : 1 \le \alpha < \alpha_0 \}$$ be a well-ordered family of its subsets. It is said that $\mathcal{U}$ is a relatively open partition of $X$ if:
    \begin{enumerate}[(i)]
        \item $U_{\alpha}$ is contained in $X \setminus \left(\bigcup\limits_{\beta < \alpha} U_{\beta}\right)$ and it is relatively open in it for every $\alpha$;
        \item $X = \bigcup\limits_{1 \le \alpha < \alpha_0} U_{\alpha}$.
    \end{enumerate}
\end{definition}

Note that each element of a relatively open partition is an intersection of an open set and a closed set, i.e. it is locally closed.

\begin{definition}[\cite{KR}]
    Let $\mathcal{U}^1$ and $\mathcal{U}^2$ be two relatively open partitions of $X$. It is said that $\mathcal{U}^2$ is a refinement of $\mathcal{U}^1$ ($\mathcal{U}^2 \prec \mathcal{U}^1$) if:
    \begin{enumerate}[(i)]
        \item for each element $U^2 \in \mathcal{U}^2$ there exists an element $U^1 \in \mathcal{U}^1$ with $U^2 \subseteq U^1$;
        \item if $\mathcal{U}^2 \ni U^2_{\beta_i} \subseteq U^1_{\alpha_i} \in \mathcal{U}^1$, $i=1,2$ with $\alpha_1 < \alpha_2$, then $\beta_1 < \beta_2$.
    \end{enumerate}
\end{definition}

For each $\varepsilon>0$  we use suitable upper semi-continuous multi-valued mappings with nonempty convex compact values as $\varepsilon$-approximations of $F$ in order to construct
$\varepsilon$-solutions of the considered differential inclusion, i.e. solutions to the inclusion $\dot x\in  F(x) + \varepsilon\overline{\mathbf{B}}$.

\begin{definition}[\cite{KR}]\label{invapr}
    Let $D$ be a locally closed subset of $\mathbb{R}^n$ and $F: \overline{D} \rightrightarrows \mathbb{R}^n$ be a multi-valued mapping with nonempty values. Let $\varepsilon>0$ be fixed and let $\mathcal{U}$ be a relatively open partition of $D$ and $\mathcal{G}=\{G_U : U \in \mathcal{U}\}$ be a family of multi-valued mappings.

    It is said that $(\mathcal{U},\mathcal{G})$ is an invariant $\varepsilon$-approximation of $F$ if:
    \begin{enumerate}[(i)]
        \item $G_{U_\alpha} : \overline{U_\alpha} \rightrightarrows \mathbb{R}^n$ is an upper semi-continuous multi-valued mapping with nonempty convex compact values for each $\alpha \in [1,\alpha_0)$;
        \item for each $\alpha \in [1,\alpha_0)$ and for each $x \in \overline{U_{\alpha}}$ it is true that $$G_{U_{\alpha}}(x) \cap T^B_{\overline{U_{\alpha}}}(x) \subseteq F(x) + \varepsilon\overline{\mathbf{B}};$$
        \item for each $\alpha \in [1,\alpha_0)$ and for each $x \in U_{\alpha}$ the intersection $G_{U_{\alpha}}(x) \cap T^B_{U_{\alpha}}(x)$ is nonempty.
    \end{enumerate}
\end{definition}

This is a slight modification of the original Definition 2.2 from \cite{KR}. All the results in \cite{KR} remain true with this modified definition since for each element of the partition the existence of tangent velocities is needed only at each point of the element and not of its closure.

The following definition is a generalization of the notion of an invariant $\varepsilon$-approximation for the case when a locally finite partition of the above kind cannot be constructed. Such an example is given in \cite{KR} as well as in the last section of this paper.

\begin{definition}[\cite{KR}]\label{invpapr}
    Let $D$ be a locally closed subset of $\mathbb{R}^n$ and $F: \overline{D} \rightrightarrows \mathbb{R}^n$ be a multi-valued mapping with nonempty values. Let $\varepsilon>0$ be fixed and let $\mathcal{U}$ be a $\sigma$-relatively open partition of $D$ (that is, $\mathcal{U}=\bigcup_{m=1}^{\infty} \mathcal{U}^m$, where each $\mathcal{U}^m$ is a relatively open partition of $D$) such that $\mathcal{U}^{m+1} \prec \mathcal{U}^m$ for all $m \in \mathbb{N}$. Let $\mathbf{Green} \, \mathcal{U}$ be a disjoint subfamily of $\mathcal{U}$ and let $\mathcal{G} = \{G_U : U \in \mathbf{Green} \, \mathcal{U}\}$ be a family of multi-valued mappings.

    It is said that $(\mathcal{U},\mathcal{G})$ is an invariant partial $\varepsilon$-approximation of $F$ if:
    \begin{enumerate}[(i)]
        \item if $U \in (\mathbf{Green} \, \mathcal{U}) \cap \mathcal{U}^m$ then $U \in \mathcal{U}^k$ for all $k \ge m$;
        \item for each point $x \in \cup (\mathbf{Green} \, \mathcal{U})$ there exists a neighborhood of $x$ which intersects at most finitely many elements of $\mathbf{Green} \, \mathcal{U}$;
        \item the set $\cup (\mathbf{Green} \, \mathcal{U})$ is open in $D$;
        \item $(\mathbf{Green} \, \mathcal{U}, \mathcal{G})$ is an invariant $\varepsilon$-approximation of $F$, i.e.
        \begin{enumerate}[(a)]
            \item $G_U : \overline{U} \rightrightarrows \mathbb{R}^n$ is an upper semi-continuous multi-valued mapping with nonempty convex compact values for each $U \in \mathbf{Green} \, \mathcal{U}$;
            \item for each $U  \in \mathbf{Green} \, \mathcal{U}  $ and for each $x \in \overline{U} $ it is true that  $$G_{U}(x) \cap T^B_{\overline{U}}(x) \subseteq F(x) + \varepsilon\overline{\mathbf{B}};$$
            \item for each $U  \in \mathbf{Green} \, \mathcal{U}  $ and for each $x \in U$ the intersection $G_{U}(x) \cap T^B_{U}(x)$ is nonempty.
            \end{enumerate}
    \end{enumerate}
\end{definition}

Roughly speaking, the meaning of Definition \ref{invpapr} is that the domain can be divided into two parts: ``good points'' (the set $\cup (\mathbf{Green} \, \mathcal{U})$) and ``bad points'' (the complement of $\cup (\mathbf{Green} \, \mathcal{U})$ in $D$). The set of ``good points'' can be partitioned into a locally finite family of subsets each of which is weakly forward invariant with respect to a suitably chosen approximating mapping. Depending on the problem, additional efforts are necessary to deal with the ``bad points''.  In \cite{KR} a scheme is suggested for proofs of the existence of solutions of differential inclusions using this structure. It could be implemented in different cases if one can handle the ``bad points'' of the right-hand side.

\begin{rem}[\cite{KR}]
    There is a natural linear order $\prec$ in the family $\mathbf{Green}\,\mathcal{U}$ inherited from the order in $\mathcal{U}_m, m = 1,2,\dots$.
\end{rem}

The following proposition is a slight modification of Proposition 2.3(ii) and (iii) from \cite{KR} that is stated for an invariant $\varepsilon$-approximation with a locally finite relatively open partition of the domain. Since $\mathbf{Green}\,\mathcal{U}$ is locally finite and open in $D$ by definition, every point of $\cup\left(\mathbf{Green}\,\mathcal{U}\right)$ has a neighborhood $V$ (in $D$) that is contained in $\cup\left(\mathbf{Green}\,\mathcal{U}\right)$ and intersects finitely many elements of $\mathbf{Green}\,\mathcal{U}$. Thus, there exists a positive integer $m$ such that $V \cap \mathcal{U}^m$ is a relatively open partition of $V \cap D$ with finitely many elements. Then we are in the situation of the proposition cited from \cite{KR}. This idea is implemented in Theorem 3.4 from the same paper.

\begin{prop}\label{ulimit-sol}
    Let $D$ be a locally closed subset of $\mathbb{R}^n$ and $F: \overline{D} \rightrightarrows \mathbb{R}^n$ be uniformly bounded. Let $\mathcal{U}$ be a $\sigma$-relatively open partition of $D$ and $(\mathcal{U},\mathcal{G})$ be an invariant partial $\varepsilon$-approximation of $F$. Then:
    \begin{enumerate}[(i)]
        \item for each absolutely continuous function $\varphi: [0,T] \to \cup\left(\mathbf{Green}\,\mathcal{U}\right)$ with respect to which the elements of $\mathbf{Green}\,\mathcal{U}$ are forward invariant, the function $$t \mapsto \alpha(t), \text{where } \varphi(t) \in U_{\alpha(t)}$$ is monotone increasing;
        \item if $\varphi_k : [0,T] \to \cup\left(\mathbf{Green}\,\mathcal{U}\right)$, $k=1,2,\dots$ is a sequence of $\varepsilon$-solutions of $\dot{x} \in F(x)$ with respect to which the elements of $\mathbf{Green}\,\mathcal{U}$ are forward invariant and $\varphi : [0,T] \to \cup\left(\mathbf{Green}\,\mathcal{U}\right)$ is its uniform limit, then $\varphi$ is an $\varepsilon$-solution of $\dot{x} \in F(x)$.
    \end{enumerate}
\end{prop}

\begin{definition}[\cite{KR}]
    Let $D$ be a locally closed subset of $\mathbb{R}^n$ and $F: \overline{D} \rightrightarrows \mathbb{R}^n$ be a multi-valued mapping with nonempty values. Let $\mathcal{U}^i$ be a $\sigma$-relatively open partition of $D$ and $(\mathcal{U}^i,\mathcal{G}^i)$ be an invariant partial $\varepsilon_i$-approximation of $F$, $i=1,2$. It is said that $(\mathcal{U}^2,\mathcal{G}^2)$ is a refinement of $(\mathcal{U}^1,\mathcal{G}^1)$ if:
    \begin{enumerate}[(i)]
        \item $0 < \varepsilon_2 < \varepsilon_1$;
        \item for each $U^2 \in \mathbf{Green} \,\mathcal{U}^2$ there exists $U^1 \in \mathbf{Green} \,\mathcal{U}^1$ with $U^2 \subseteq U^1$;
        \item if $\mathbf{Green} \,\mathcal{U}^2 \ni U^2_i \subseteq U^1_i \in \mathbf{Green} \,\mathcal{U}^1$, $i=1,2$ with $U^1_2 \prec U^1_1$, then $U^2_2 \prec U^2_1$;
        \item if $\mathbf{Green} \,\mathcal{U}^2 \ni U^2 \subseteq U^1 \in \mathbf{Green} \,\mathcal{U}^1$, then $G_{U^2}(x) \subseteq G_{U^1}(x)$ for each $x \in \overline{U^2}$.
    \end{enumerate}
\end{definition}

These concepts have been used in \cite{KR} to study the existence of solutions of differential inclusions of the considered form. In the next section, we will make use of them in order to investigate further this problem.

\section{Main result}\label{main}

The following definition is motivated by the concept of stability and the related notion of a Lyapunov function. The function we define strictly increases at least by a constant rate along the trajectories, so they ``move away'' from the set of ``bad points''. We require continuity on the whole domain and smoothness on each element of the partition of the set of ``good points''.

\begin{definition}\label{quasilyapunov}
    Let $D$ be a locally closed subset of $\mathbb{R}^n$ and $F: \overline{D} \rightrightarrows \mathbb{R}^n$ be a multi-valued mapping with nonempty values. Let $\mathcal{U}$ be a $\sigma$-relatively open partition of $D$ and $(\mathcal{U},\mathcal{G})$ be an invariant partial $\varepsilon$-approximation of $F$. It is said that $w: D \to \mathbb{R}$ is a quasi-Lyapunov function for $(\mathcal{U},\mathcal{G})$ if:
    \begin{enumerate}[(i)]
        \item $w(x) \ge 0$ for all $x \in D$ and $w(x) = 0$ iff $x \notin \cup(\mathbf{Green} \,\mathcal{U})$;
        \item $w$ is a continuous function;
        \item the restriction $w\!\restriction_{U}$ of $w$ to $U$ is a smooth function for all $U \in \mathbf{Green} \,\mathcal{U}$ (i.e. it is a restriction of a smooth function which is defined on an open set containing $U$);
        \item for each $U \in \mathbf{Green} \,\mathcal{U}$ and for each $x \in U$ the following inequality holds: $$\max_{y \in G_U(x) \cap T^B_{U}(x)} \langle \mathbf{grad} \, w(x) , y \rangle \ge 1.$$
    \end{enumerate}
\end{definition}

The maximum from condition $(iv)$ of the above definition is required to be uniformly separated from zero by a positive number. Without loss of generality, we assume that this number is equal to $1$. This guarantees that $\varepsilon$-solutions starting from ``good points'' always ``stay away'' from the ``bad points'' as well as uniform limits of such $\varepsilon$-solutions. In this way, using Proposition \ref{ulimit-sol}(ii), we can prove the existence of $\varepsilon$-solutions starting from the boundary of the set of ``bad points''.

\begin{prop}\label{esolution}
    Let $D$ be an intersection of an open subset $S$ and a closed subset $K$ of $\mathbb{R}^n$ and $F : \overline{D} \rightrightarrows \mathbb{R}^n$ be a uniformly bounded multi-valued mapping with nonempty values.
    Let $\mathcal{U}$ be a $\sigma$-relatively open partition of $D$ and $(\mathcal{U},\mathcal{G})$ be an invariant partial $\varepsilon$-approximation of $F$. Let $w: D \to \mathbb{R}$ be a quasi-Lyapunov function for $(\mathcal{U},\mathcal{G})$.

    Then for each point $x_0 \in \overline{\cup(\mathbf{Green} \,\mathcal{U})} \cap D$ there exists an $\varepsilon$-solution $\varphi(\cdot)$ of $\dot{x} \in F(x)$ starting from $x_0$ and defined on the interval $[0,T)$ ($T=+\infty$ or $\varphi(t)$ tends to a point in $\mathbb{R}^n \setminus S$ whenever $t \to T$) such that $w(\varphi(t)) \ge w(x_0) + t$ for all $t \in [0,T)$. Moreover, if the starting point $x_0$ belongs to $\cup(\mathbf{Green} \,\mathcal{U})$, then each element of $\mathbf{Green} \,\mathcal{U}$ is forward invariant with respect to the corresponding $\varepsilon$-solution $\varphi(\cdot)$ starting from $x_0$.
\end{prop}

\begin{proof}
    First, we consider the case where $x_0 \in \cup(\mathbf{Green} \,\mathcal{U})$.

    For each point $x \in \cup(\mathbf{Green} \,\mathcal{U})$ there exists a unique index $\alpha_x$ such that $x \in U_{\alpha_x}$. We define a multi-valued mapping $\widetilde{G}_{U_{\alpha_x}} : U_{\alpha_x} \rightrightarrows \mathbb{R}^n$, $$\widetilde{G}_{U_{\alpha_x}} (z) := \left\{ y \in G_{U_{\alpha_x}}(z) : \langle \mathbf{grad} \,w(z), y \rangle \ge 1\right\}.$$
    Since $\widetilde{G}_{U_{\alpha_x}}(z) \subseteq G_{U_{\alpha_x}}(z)$ for all $z \in U_{\alpha_x}$, every solution of the inclusion with right-hand side $\widetilde{G}_{U_{\alpha_x}}$ is a solution of the inclusion with right-hand side $G_{U_{\alpha_x}}$. Using this, we will construct an $\varepsilon$-solution of $\dot{x} \in F(x)$ such that $w$ increases at least by a constant rate along it.

    Since $G_{U_{\alpha_x}}$ has nonempty convex compact values and is uniformly bounded (because $F$ is uniformly bounded) and Definition \ref{quasilyapunov}(iv) holds, the same is true for the mapping $\widetilde{G}_{U_{\alpha_x}}$. It is also straightforward to check that $\widetilde{G}_{U_{\alpha_x}}$ is upper semi-continuous on its domain $U_{\alpha_x}$.

    Moreover, $\widetilde{G}_{U_{\alpha_x}}(z) \cap T^B_{U_{\alpha_x}} (z) \ne \emptyset$ for all $z \in U_{\alpha_x}$ because of Definition \ref{quasilyapunov}(iv).
    Then, applying Theorem \ref{localsol}, we obtain the existence of $t_x >0$ and an absolutely continuous function $\varphi_{x,\tau} : [\tau,\tau+t_x) \to U_{\alpha_x}$, which is a solution of the differential inclusion
    $$\left| \begin{array}{l}
    \dot{x}(t) \in \widetilde{G}_{U_{\alpha_x}}(x(t)) \subseteq G_{U_{\alpha_x}}(x(t)) \text{ a.e. on } [\tau,\tau+t_x)\\
    x(\tau) = x\\
    \end{array} \right..$$

    Since $\varphi_{x,\tau}(t) \in U_{\alpha_x}$ for all $t \in [\tau,\tau+t_x)$, then
    $$\dot{\varphi}_{x,\tau}(t) \in \widetilde G_{U_{\alpha_x}}(\varphi_{x,\tau}(t)) \cap T^B_{U_{\alpha_x}} (\varphi_{x,\tau}(t)) \subseteq F(\varphi_{x,\tau}(t)) + \varepsilon \overline{\mathbf{B}} \text{ a.e. on } [\tau,\tau+t_x).$$

    According to the definition of the mapping $\widetilde{G}_{U_{\alpha_x}}$, we have
    $$\frac{d}{dt} w(\varphi_{x,\tau}(t)) = \langle \mathbf{grad} \,w(\varphi_{x,\tau}(t)), \dot{\varphi}_{x,\tau}(t) \rangle \ge 1 \text{ for almost every }t \in [\tau,\tau+t_x).$$
    Since $w(\varphi_{x,\tau}(\cdot))$ is absolutely continuous on $[\tau,t]$ for all $t \in [\tau,\tau + t_x)$ as a composition of a smooth function and an absolutely continuous function defined on a compact interval, after integration we obtain that $$w(\varphi_{x,\tau}(t)) \ge w(x) + (t-\tau) \text{ for all }t \in [\tau,\tau+t_x).$$

    In particular, for the point $x_0$ there exist $t_{x_0}>0$ and an absolutely continuous function $\varphi_{x_0,0} : [0,t_{x_0}) \to U_{\alpha_{x_0}}$. We inductively define an absolutely continuous function $\varphi : [0,T) \to D$.
    Let $\varphi(t) := \varphi_{x_0,0}(t)$ for all $t \in [0,t_{x_0})$. Assume that $\varphi$ is defined on some interval $[0,\hat{t})$ and $w(\varphi(t)) \ge w(x_0) + t$ for all $t \in [0,\hat{t})$.
    Then for each increasing sequence $\{t_n\}_{n=1}^{\infty}, t_n \xrightarrow[]{} \hat{t}$, by using the uniform boundedness of the values of $F$ by $c$ and assuming that $\varepsilon \in (0,1)$, we have
    $$\left\|\varphi(t_n)-\varphi(t_{n-1})\right\| \le \int_{t_{n-1}}^{t_n} \left\|\dot{\varphi}(\tau)\right\|\mathrm{d}\tau \le (c+1)(t_n-t_{n-1}).$$
    This means that the sequence $\left\{\varphi(t_n)\right\}_{n=1}^{\infty}$ is a Cauchy sequence. Hence, there exists $\varphi(\hat{t}) := \lim_{t \to \hat{t}} \varphi(t)$.
    If $\varphi(\hat{t}) \in D$, then $w(\varphi(\hat{t})) \ge w(x_0) + \hat{t}$ by taking a limit as $t \to \hat{t}$ and using the continuity of $w$. Then from Definition \ref{quasilyapunov}(i) and $w(\varphi(\hat{t})) > 0$ we have $\varphi(\hat{t}) \in \cup(\mathbf{Green} \,\mathcal{U})$. Hence, we can define an absolutely continuous extension of the function $\varphi$ by $\varphi_{\varphi(\hat{t}),\hat{t}}$ on the interval $\left[\hat{t},\hat{t}+t_{\varphi(\hat{t})}\right)$.
    Moreover, for each $t \in \left[\hat{t},\hat{t}+t_{\varphi(\hat{t})}\right)$ we have
    \begin{align*}
    w(\varphi(t)) &= w(\varphi_{\varphi(\hat{t}),\hat{t}}(t)) \ge w(\varphi_{\varphi(\hat{t}),\hat{t}}(\hat{t})) + (t-\hat{t})\\
    &= w(\varphi(\hat{t}))+ (t-\hat{t})
    \ge w(x_0) + t.
    \end{align*}

    Let $[0,T_0)$ be the maximal interval on which $\varphi$ can be defined this way.
    If $T_0 \ne +\infty$, let $\varphi(t) \xrightarrow[]{} y$ whenever $t \to T_0$ (the limit exists as is seen above).
    If $y \in D$, then $y \in \cup(\mathbf{Green} \,\mathcal{U})$ by following the same arguments as above, so we can define an absolutely continuous extension of $\varphi$ in the same way as above. This contradicts the maximality of the interval.
    Then $y \notin D$. Since $D = S \cap K$, where $S$ is an open subset and $K$ is a closed subset of $\mathbb{R}^n$, the only possibility is that $\varphi(t)$ tends to the boundary of $S$ whenever $t \to T_0$, i.e. $y \in \mathbb{R}^n \setminus S$.
    For each $t \in [0,T_0)$ we have $$\left\|\varphi(t)-x_0\right\| = \left\|\int_{0}^{t} \dot{\varphi}(\tau)\mathrm{d}\tau\right\| \le \int_{0}^{t} \left\|\dot{\varphi}(\tau)\right\|\mathrm{d}\tau \le (c+1)t.$$
    Then, by taking a limit as $t \to T_0$ in $\left\|\varphi(t)-x_0\right\| \le (c+1)t$, we obtain that $\|y-x_0\| \le (c+1)T_0$. Since $y \in \mathbb{R}^n \setminus S$, $\mathrm{dist}\left(x_0,\mathbb{R}^n \setminus S\right) \le \|y-x_0\| \le (c+1)T_0$. Since $S$ is an open subset of $\mathbb{R}^n$, we have
    \begin{equation}\label{time}
        T_0 \ge \frac{\mathrm{dist}(x_0,\mathbb{R}^n \setminus S)}{c+1} > 0.
    \end{equation}

    Now, let $x_0 \in \overline{\cup(\mathbf{Green} \,\mathcal{U})} \cap D$. Then there exists a sequence of points $\{x_k\}_{k=1}^{\infty} \subseteq \cup(\mathbf{Green} \,\mathcal{U})$ such that $x_k \xrightarrow[k \to \infty]{} x_0$.
    For each positive integer $k$ the first part of the proof implies the existence of an absolutely continuous function $\varphi_k : [0,T_k) \to D$ such that $\varphi_k$ is an $\varepsilon$-solution of $\dot{x} \in F(x)$ starting from $x_k$, each element of $\mathbf{Green} \,\mathcal{U}$ is forward invariant with respect to $\varphi_k$ and $w(\varphi_k(t)) \ge w(x_k) + t$ for all $t \in [0,T_k)$.

    The convergence of the sequence $\{x_k\}_{k=1}^{\infty}$ implies the existence of a constant $k_0 \in \mathbb{N}$ such that for all $k \ge k_0$ we have $\|x_k-x_0\| \le \frac{1}{2} \mathrm{dist}(x_0,\mathbb{R}^n \setminus S)$. Using the lipschitzness of the distance function and the fact that $S$ is an open subset of $\mathbb{R}^n$, for each $k \ge k_0$ we have (as the estimate for $T_0$ in (\ref{time}))
    \begin{equation}\label{time2}
        \begin{aligned}
            T_k   {\ge} \frac{\mathrm{dist}(x_k,\mathbb{R}^n \setminus S)}{c+1} &\ge \frac{\mathrm{dist}(x_0,\mathbb{R}^n \setminus S) - \|x_k-x_0\|}{c+1}\\ &\ge T:= \frac{\mathrm{dist}(x_0,\mathbb{R}^n \setminus S)}{2(c+1)} > 0.
        \end{aligned}
    \end{equation}

    Since $F$ is uniformly bounded, the sequence $\{\varphi_k(\cdot)\}_{k=k_0}^{\infty}$ of functions defined on the interval $[0,T]$ is uniformly bounded and equicontinuous. Applying the Arzelà–Ascoli theorem, without loss of generality, we may assume that the sequence is uniformly convergent to an absolutely continuous function $\varphi : [0,T] \to D$.

    Since $w$ is continuous, by taking the limit as $k \to \infty$ in the inequality $w(\varphi_k(t)) \ge w(x_k) + t$, we have $w(\varphi(t)) \ge w(x_0) + t$ for all $t \in [0,T]$. Hence, $w(\varphi(t)) > 0$ for all $t \in (0,T]$. According to Definition \ref{quasilyapunov}, $\varphi(t) \in \cup(\mathbf{Green} \,\mathcal{U})$ for all $t \in (0,T]$. Then Proposition \ref{ulimit-sol}(ii) implies that $\varphi$ is an $\varepsilon$-solution of $\dot{x} \in F(x)$ starting from $x_0$ and defined on $[0,T]$.

    Following the inductive construction of an $\varepsilon$-solution from the first part of the proof for the point $\varphi(T) \in \cup(\mathbf{Green} \,\mathcal{U})$ instead for the point $x_0$, we find a maximal interval $\left[0,\overline{T}\right)$ on which the solution is defined such that $\overline{T}=+\infty$ or $\varphi(t)$ tends to a point in $\mathbb{R}^n \setminus S$ whenever $t \to \overline{T}$.
\end{proof}

Assuming closedness of the values of the right-hand side and existence of successive refinements of the invariant partial $\varepsilon$-approximation, we can pass to the limit as $\varepsilon \to 0$ and obtain a solution of the considered differential inclusion.

\begin{thm}\label{sol}
    Let $D$ be a locally closed subset of $\mathbb{R}^n$ and $F : \overline{D} \rightrightarrows \mathbb{R}^n$ be a uniformly bounded multi-valued mapping with nonempty closed values. Let $\mathcal{U}^k$ be a $\sigma$-relatively open partition of $D$, $(\mathcal{U}^k,\mathcal{G}^k)$ be an invariant partial $\frac{1}{k}$-approximation of $F$ and $(\mathcal{U}^{k+1},\mathcal{G}^{k+1})$ be a refinement of $(\mathcal{U}^k,\mathcal{G}^k)$ for all $k=1,2,\dots$. We assume that for each approximation $(\mathcal{U}^k,\mathcal{G}^k)$ there exists a quasi-Lyapunov function $w_k : D \to \mathbb{R}$ such that $w_{k_1} \ge w_{k_2}$ on $D$ for every $k_1 \le k_2$.

    Then for each point $x_0 \in \bigcap_{k=1}^{\infty}\left(\overline{\bigcup(\mathbf{Green} \,\mathcal{U}^k)}\right) \bigcap D$ there exist $T>0$ and an absolutely continuous function $x : [0,T] \to D$ such that
    $$\left| \begin{array}{l}
    \dot{x}(t) \in F(x(t)) \text{ a.e. on } [0,T]\\
    x(0) = x_0\\
    \end{array} \right..$$
\end{thm}
\begin{proof}
    Fix an arbitrary point $x_0 \in \bigcap_{k=1}^{\infty}\left(\overline{\bigcup(\mathbf{Green} \,\mathcal{U}^k)}\right) \bigcap D$ and a positive integer $k$. Then $x_0 \in \left(\overline{\bigcup(\mathbf{Green} \,\mathcal{U}^k)} \cap D\right)$, so there exists $x_k \in \bigcup(\mathbf{Green} \,\mathcal{U}^k)$ such that $\|x_k-x_0\| < \frac{1}{k}$.

    Then, by applying Proposition \ref{esolution}, there exists a $\frac{1}{k}$-solution $\varphi_k(\cdot)$ of $\dot{x} \in F(x)$ starting from $x_k$ and defined on some maximal interval $[0,T_k)$ such that each element of $\mathbf{Green} \,\mathcal{U}^k$ is forward invariant with respect to $\varphi_k$ and $w_k(\varphi_k(t)) \ge w_k(x_k) + t$ for all $t \in [0,T_k)$.

    Let $D = S \cap K$, where $S$ is an open subset and $K$ is a closed subset of $\mathbb{R}^n$. According to Proposition \ref{esolution}, $T_k = +\infty$ or $\varphi_k(t)$ tends to a point in $\mathbb{R}^n \setminus S$ whenever $t \to T_k$. According to $(\ref{time})$, $$T_k \ge \frac{\mathrm{dist}(x_k,\mathbb{R}^n \setminus S)}{c+1} > 0.$$

    Repeating this for each positive integer $k$, we obtain a sequence $\{x_k\}_{k=1}^{\infty}$ which tends to $x_0$. Thus, there exists a constant $k_0 \in \mathbb{N}$ such that for all $k \ge k_0$ we have $\|x_k-x_0\| \le \frac{1}{2} \mathrm{dist}(x_0,\mathbb{R}^n \setminus S)$. Using the same arguments as in $(\ref{time2})$, for each $k \ge k_0$ we get $$T_k \ge T:= \frac{\mathrm{dist}(x_0,\mathbb{R}^n \setminus S)}{2(c+1)} > 0.$$

    Without loss of generality, we may think that the sequence $\{\varphi_k(\cdot)\}_{k=k_0}^{\infty}$ is uniformly convergent on the interval $[0,T]$ to an absolutely continuous function $\varphi : [0,T] \to D$.

    Fix an arbitrary positive integer $m \ge k_0$. Since $(\mathcal{U}^k,\mathcal{G}^k)$ is a refinement of $(\mathcal{U}^m,\mathcal{G}^m)$ for each $k \ge m$, the elements of $\mathbf{Green} \,\mathcal{U}^m$ are forward invariant with respect to the trajectory $\varphi_k(\cdot)$ for each $k \ge m$. Since $w_m(\varphi_k(t)) \ge w_k(\varphi_k(t)) \ge w_k(x_k) + t \ge t$ for all $t \in [0,T]$, by taking the limit as $k \to +\infty$, we obtain that $w_m(\varphi(t)) \ge t$ for all $t \in [0,T]$. Thus, $\varphi(t) \in \cup(\mathbf{Green} \,\mathcal{U}^m)$ for all $t \in (0,T]$. According to Proposition \ref{ulimit-sol}(ii), $\varphi(\cdot)$ is a $\frac{1}{m}$-solution of $\dot{x} \in F(x)$, i.e. $$\dot{\varphi}(t) \in F(\varphi(t)) + \frac{1}{m} \overline{\mathbf{B}} \text{ a.e. on } [0,T].$$

    Since $\dot{\varphi}(t)$ does not depend on $m$, the multi-valued mapping $F$ is closed-valued and the union of countably many sets of measure zero is a set of measure zero, we obtain that $$\dot{\varphi}(t) \in F(\varphi(t)) \text{ a.e. on } [0,T].$$

    This completes the proof.
\end{proof}

The structure of an invariant partial $\varepsilon$-approximation of a multi-valued mapping has been used in \cite{KR} to obtain the existence of solutions of inclusions of the form $\dot{x} \in F(x)+G(x)$, where $F$ is lower semi-continuous on a $G_{\delta}$ subset of its domain $D$ and upper semi-continuous and convex-valued on the complement and $G$ is upper semi-continuous and convex-valued on its domain $D$.  One could easily adapt the proof of Theorem 4.1 from \cite{KR} (with ${G}$ equal to the zero vector of $\mathbb{R}^{n}$) replacing the convexity assumption
on the values of $F$ at the ``bad points'' with the existence of a suitable Lyapunov-like function. Thus, using Theorem \ref{sol} from the present paper, we obtain the following

\begin{corol}\label{mixed}
    Let $D$ be a locally closed subset of $\mathbb{R}^n$ and $F$ be a uniformly bounded mapping with nonempty closed values defined on $\overline{D}$. We assume that there exist a sequence of positive reals $\varepsilon_k$, $k=1,2,\dots$ which tends to zero as $k$ tends to infinity,  relatively open and dense (in $D$) sets $D_k$ such that $F$ is $\varepsilon_k$-lower semi-continuous on $D_k$ (we assume that $D_1 = D$) and there exists a sequence of continuous functions $w_k :  {D} \to \mathbb{R}$ such that:
     \begin{enumerate}[(i)]
           \item $w_{k_1} \ge w_{k_2}$ on $ {D}$ for every $k_1 \le k_2$;
           \item for each positive integer $k$ we have $w_k(x) \ge 0$ for all $x \in D$ and $w_k(x) = 0$ iff $x \notin D_k$;
           \item for each positive integer $k$ the restriction $w_k\!\restriction_{D_k}$ of $w_k$ to $D_k$ is a smooth function;
          \item for each positive integer $k$ and for each $x \in D_k$ the set $F_k (x) $ is nonempty, where $F_0 (x):=F(x)$ and
          for each $k\ge 1$
          $$F_k(x):= \{y \in F_{k-1}(x):\    \langle \mathbf{grad} \, w_k(x) , y \rangle \ge 1\}.$$
     \end{enumerate}
    We assume that for each positive integer $k$ and for each $x \in D_k$ the inclusion $F(x) \subseteq T^B_{D}(x)\cap F_k(x) + \varepsilon_k \overline{\textbf{B}}$ holds.

    Then for each point $x_0 \in  D$ the differential inclusion $\dot{x} \in F(x)$ has a solution starting from $x_0$.
\end{corol}

\section{An example}\label{example}

The example in this section originates from the Fuller problem from optimal control theory (cf., e.g., Chapter 2 from \cite{ZB}). Given an arbitrary $T>0$, consider the following control system in the plane $$\dot{x}(t) = y(t),\quad \dot{y}(t) = u(t).$$ The admissible controls are all measurable functions $u:[0,T] \to [-1,1]$.  The optimal control problem is to minimize  the cost functional $$\int\limits_{0}^{T} |x(t)|^q \mathrm{d}t \quad (q > 1)$$ over all trajectories of the control system, satisfying the boundary conditions $$x(0)=0,\quad y(0)=0,\quad x(T)=x_0,\quad y(T)=y_0.$$

This problem has been analyzed in Lemma 2.2 and Lemma 2.3 from \cite{ZB}. For every final state $(x_0,y_0)$ there exists $T_0 > 0$ such that for all $T>T_0$ the problem has a unique optimal solution with a control function with a countable number of switching times. The switching points of the corresponding optimal trajectory determine a convergent sequence which tends to the origin. The curve $x=C(q)y^2\sgn{y}$ is the switching set. The optimal control equals $1$ on the left-hand side of the switching curve and equals $-1$ on the right-hand side. For every $q>1$ the parameter $C(q)$ satisfies the inequality $0 < C(q) < \frac{1}{2}$.

Taking into account the solution of this problem, we consider a differential inclusion with the following upper semi-continuous right-hand side:

\begin{equation*}
    F(x,y) = \begin{cases}
        \{(y,1)\}, &\text{if } x < \frac{1}{4}y^2 \sgn y\\
        \{(y,-1),(y,1)\}, &\text{if } x = \frac{1}{4}y^2, y>0\\
        \{(0,1),(0,-1)\}, &\text{if } x=0,y=0\\
        \{(y,-1),(y,1)\}, &\text{if } x = -\frac{1}{4}y^2, y<0\\
        \{(y,-1)\}, &\text{if } x > \frac{1}{4}y^2 \sgn y\\
    \end{cases}.
\end{equation*}

To simplify the calculations, we have chosen the value $\frac{1}{4}$ for the parameter $C(q)$. The same kind of reasoning can be made for every value of $C(q)$ in the interval $\left(0,\frac{1}{2}\right)$.

The existence of a solution of the differential inclusion
\begin{equation}\label{diffincl}
    (\dot{x},\dot{y}) \in F(x,y), x(0)=x_0, y(0)=y_0,
\end{equation} for each initial condition $(x_0,y_0)$ can be proved by using Lemma 2.3 from \cite{ZB}. In the current section, we propose a different way to obtain the existence of a solution of this inclusion by using the result from Section \ref{main}. For this purpose, we will construct an invariant partial $\varepsilon$-approximation of $F$, which does not depend on the choice of $\varepsilon > 0$, and a quasi-Lyapunov function for this approximation. Then we will apply Theorem \ref{sol}. Note that we cannot apply Corollary \ref{mixed} because its assumptions concerning the Lyapunov-like function cannot be satisfied.

We define the following sets
\begin{align*}
   D_1 &= \left\{(x,y) \in \mathbb{R}^2 : x \le \frac{1}{4}y^2 \sgn y\right\},\quad D_2 = \left\{(x,y) \in \mathbb{R}^2 : x \ge \frac{1}{4}y^2 \sgn y\right\},\quad \\
C_1 &= \left\{(x,y) \in \mathbb{R}^2 : x = \frac{1}{4}y^2, y \ge 0\right\},\quad C_2 = \left\{(x,y) \in \mathbb{R}^2 : x = -\frac{1}{4}y^2 , y \le 0\right\}.
\end{align*}
It is easy to check that the sets $D_1 \setminus C_1$ and $D_2 \setminus C_2$ are weakly forward invariant with respect to $F$. Thus, the existence of a solution of the differential inclusion $(\ref{diffincl})$ is guaranteed for each initial condition except for the origin. Therefore, we will consider the differential inclusion $(\ref{diffincl})$ in a neighborhood $V=r\mathbf{B}$ of the origin. Due to the weak forward invariance of the sets $D_1 \setminus C_1$ and $D_2 \setminus C_2$, we expect that the set of ``good points'' is the intersection of $\left(D_1 \setminus C_1\right) \cup \left(D_2 \setminus C_2\right)$ with $V$, i.e. the set $V \setminus \{(0,0)\}$. In order to build an invariant partial $\varepsilon$-approximation, we will add one new variable representing the time and consider the problem in $\mathbb{R}^3$. The details follow.

Let $T > 0$ be an arbitrary positive integer. We set $\widetilde{D} = V \times [0,T)$, which is an intersection of an open subset and a closed subset of $\mathbb{R}^3$. We set $z=(x,y,t) \in \mathbb{R}^3, \widetilde{F}(z) = F(x,y) \times \{1\}$ and consider the differential inclusion $$\dot{z} \in \widetilde{F}(z), z \in \widetilde{D}.$$

We will use the so-called ``ice-cream cone'' (cf. \cite{BC}, \cite{D}) $$K_{\delta} = \left\{(x,y,t) \in \mathbb{R}^3 : \|(x,y)\| \le ct, t \in [0,\delta)\right\},$$ where $\delta > 0$ and $c = \sup\left\{\|v\| : (x,y) \in V, v \in F(x,y)\right\}$. The set $K_{\delta}$ has been used by Bressan and Cortesi in their proof for the existence of a solution of a differential inclusion with a lower semi-continuous right-hand side (cf. \cite{BC}).

Since $F$ is uniformly bounded by $c$, the ``ice-cream cones'' $z+K_\delta$, $z \in \widetilde{D}$ are strongly forward invariant with respect to $\widetilde{F}$.

Let $\widetilde{D_1} = (D_1 \cap V) \times [0,T)$, $\widetilde{D_2} = (D_2 \cap V) \times [0,T)$, $\widetilde{C_1} = (C_1 \cap V) \times [0,T)$, $\widetilde{C_2} = (C_2 \cap V) \times [0,T)$ and $B = \{(0,0)\} \times [0,T)$.
We modify the construction from \cite{KR} by splitting the ``ice-cream cones'' which pass through $\widetilde{C_1} \cup \widetilde{C_2}$, otherwise $\widetilde{F}$ can't be approximated on the whole ``ice-cream cone'' with an upper semi-continuous mapping with convex compact values (cf. Definition \ref{invpapr}(iv.a)) because $\widetilde{F}$ is not convex-valued on $\widetilde{C_1} \cup \widetilde{C_2}$.

For each positive integer $s > \frac{2}{r}$ let us consider the set
$$M_s = \left\{ (x,y) \in V : \frac{1}{s} \le \|(x,y)\| \le r-\frac{1}{s}\right\} \times [0,T).$$
For each such $s$ we will construct a relatively open partition $\mathcal{V}^s$ of $\widetilde{D}$. The family $\{M_s : s \in \mathbb{N}, s > \frac{2}{r} \}$ of sets is increasing with respect to inclusion and its union is $\widetilde{D} \setminus B$.
For each $z \in \overline{M_s}$ there exists a relative neighborhood $W(z)$ of $z$ (with respect to $\overline{M_s}$) which has the form $\overline{M_s} \cap (\widetilde{z}+K_{\delta})$, where $\widetilde{z} \in (V \setminus \{(0,0)\}) \times [\tau,T)$ for $\tau < 0$ and $z$ is in the relative interior of $W(z)$ with respect to $\overline{M_s}$.
We choose $\delta > 0$ to be so small that $\overline{\widetilde{z} + K_{\delta}}$ is contained in $\overline{\widetilde{D}}$ and intersects at most one of the surfaces $\overline{\widetilde{C_1}}$ and $\overline{\widetilde{C_2}}$.

The compactness of $\overline{M_s}$ implies the existence of finitely many $z_i$, $i \in \{1,\dots,m(s)\}$ with $\overline{M_s} = \bigcup_{i=1}^{m(s)} W\left(z_i\right)$, where $W\left(z_i\right) = \overline{M_s} \cap \left(\widetilde{z}_i + K_{\delta_i}\right)$ with $\widetilde{z}_i=\left(\widetilde{x}_i,\widetilde{y}_i,\widetilde{t}_i\right)$ for each $i \in \{1,\dots,m(s)\}$.

Now, we split the ``ice-cream cones'' which pass through $\widetilde{C_1} \cup \widetilde{C_2}$. Let $P_{1,i}= \left(\widetilde{z}_i + K_{\delta_i}\right) \cap \left(\widetilde{D_1} \setminus \widetilde{C_1}\right)$ and $P_{2,i}= \left(\widetilde{z}_i + K_{\delta_i}\right) \cap \left(\widetilde{D_2} \setminus \widetilde{C_2}\right)$ for each $i \in \{1,\dots,m(s)\}$. We consider only these sets $P_{k,i}$, where $k \in \{1,2\}$ and $i \in \{1,\dots,m(s)\}$, which are nonempty and denote them by $Z_j$ for $j \in \{1,\dots,j(s)\}$. We renumber the sets $Z_j$, $j \in \{1,\dots,j(s)\}$ in the following way: if $Z_k \cap \left(\widetilde{C_1} \cup \widetilde{C_2}\right) = \emptyset$ and $Z_l \cap \left(\widetilde{C_1} \cup \widetilde{C_2}\right) \ne \emptyset$, then $k < l$ (i.e. each set which has a nonempty intersection with $\widetilde{C_1} \cup \widetilde{C_2}$ is after each set with an empty intersection with $\widetilde{C_1} \cup \widetilde{C_2}$).

We note that each of the sets $Z_j$, $j \in \{1,\dots,j(s)\}$ can intersect at most one of the surfaces $\widetilde{C_1}$ and $\widetilde{C_2}$ because of the chosen cones' heights.

We also note that each of the sets $Z_j$, $j \in \{1,\dots,j(s)\}$ is weakly forward invariant with respect to $\widetilde{F}$ as an intersection of a strongly forward invariant set (the ``ice-cream cone'') and a weakly forward invariant set ($\widetilde{D_1} \setminus \widetilde{C_1}$ or $\widetilde{D_2} \setminus \widetilde{C_2}$ - direct calculation). Each of these sets is a ``split ice-cream cone'' or an ``ice-cream cone''.

We consider the positive reals $\{t_i = \widetilde{t}_i + \delta_i : i = 1,2,\dots,m(s)\}$. Without loss of generality, we may assume that this finite sequence is non-decreasing.

We set the first element of $\mathcal{V}^s$ to be $$V^s_1 = \left\{z=(x,y,t) \in \widetilde{D} \setminus B : t<t_1\right\} \setminus \left(\bigcup_{j=1}^{j(s)} Z_{j}\right).$$
The set $V^s_1$ is relatively open in $\widetilde{D}$ because the union of the sets $Z_{j}, j \in \{1,2,\dots,j(s)\}$ intersected with the set $S=\left\{z=(x,y,t) \in \widetilde{D} \setminus B : t<t_1\right\}$ is relatively closed in $S$.
Next, we successively define $$V^s_{j+1} = \left\{z=(x,y,t) \in \widetilde{D} \setminus B : t<t_1\right\} \bigcap \left(Z_{j} \setminus \left(\bigcup_{l=j+1}^{j(s)} Z_{l}\right)\right)$$ for each $j \in \{1,2,\dots,j(s)\}$. These sets are the remaining elements of $\mathcal{V}^s$ which are contained in the first strip $\{z=(x,y,t) : t \in [0,t_1)\}$.

For each $j \in \{1,2,\dots,j(s)\}$ the set $V^s_{j+1}$ is relatively open in $\left(\widetilde{D} \setminus \left(\bigcup_{l \le j} V^s_l\right)\right) \cap S$  because $\left(\bigcup_{l=j+1}^{j(s)} Z_{l}\right) \cap S$ is relatively closed in $S$. Indeed, let $l$ be an arbitrary positive integer
belonging to $\{j+1, \dots, j(s)\}$. If $Z_l$ is not a ``split ice-cream cone'', then it is relatively closed in $S$. If $Z_l$ is a ``split ice-cream cone'', then two cases are possible: if $Z_l \cap \left(\widetilde{C_1} \cup \widetilde{C_2}\right) \ne \emptyset$, then
it is relatively closed in $S$; if $Z_l \cap \left(\widetilde{C_1} \cup \widetilde{C_2}\right) = \emptyset$, then there is a positive integer $k$ with $l<k\le j(s) $ such that $Z_l \cup Z_k$ is
the intersection of an ``ice-cream cone'' with $S$, hence it is relatively closed in $S$.

\begin{figure}[htb]
    \centering
    \includegraphics[scale=0.5]{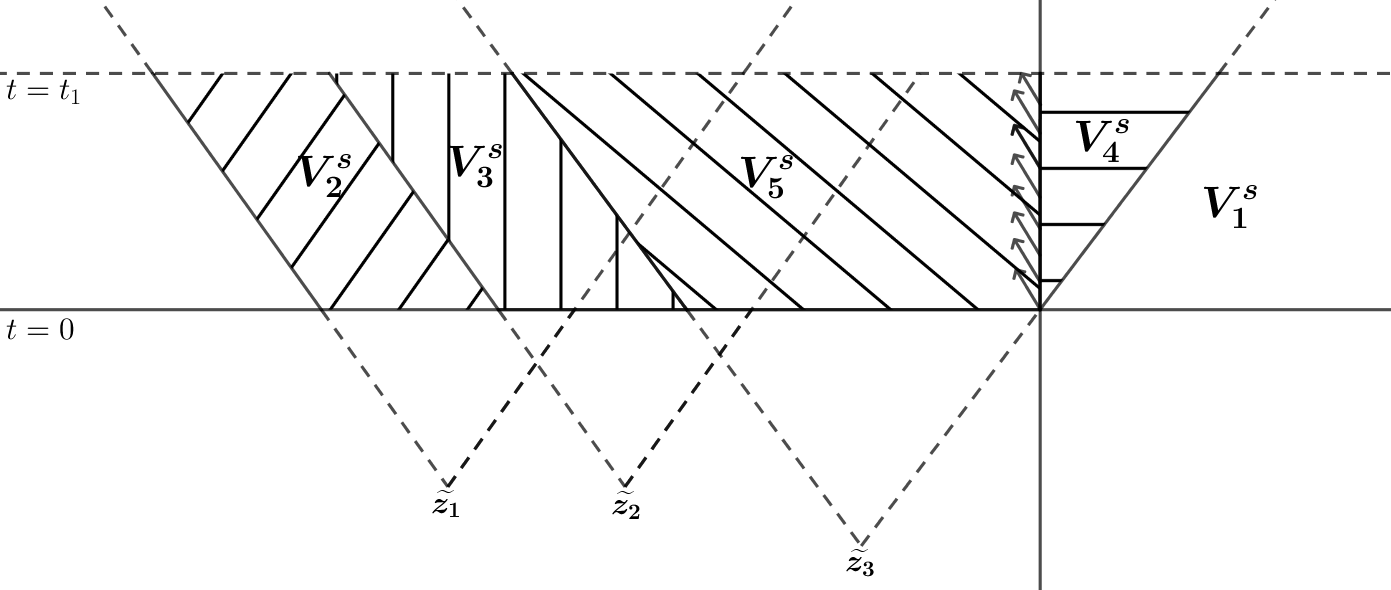}
    \caption{This figure is an attempt to give some intuition on the construction of the elements of the relatively open partition $\mathcal{V}^s$ of $\widetilde{D}$ which are contained in the first strip. Let us assume that the first strip is contained in the union of $\tilde z_i + K_{\delta_i}$, $i=1,2,3$. The ``ice-cream cones'' $\widetilde{z_1}+K_{\delta_1}$ and $\widetilde{z_2}+K_{\delta_2}$ are entirely contained in $\widetilde{D_1}\setminus\widetilde{C_1}$, so we can set $Z_1 = \widetilde{z_1}+K_{\delta_1}$ and $Z_2 = \widetilde{z_2}+K_{\delta_2}$. The ``ice-cream cone'' $\widetilde{z_3}+K_{\delta_3}$ intersects $\widetilde{C_2}$, so we split it into $\left(\widetilde{z_3}+K_{\delta_3}\right) \cap \left(\widetilde{D_1} \setminus \widetilde{C_1}\right)$ and $\left(\widetilde{z_3}+K_{\delta_3}\right) \cap \left(\widetilde{D_2} \setminus \widetilde{C_2}\right)$. Since a trajectory of $\widetilde{F}$ can pass from $\left(\widetilde{z_3}+K_{\delta_3}\right) \cap \left(\widetilde{D_2} \setminus \widetilde{C_2}\right)$ to $\left(\widetilde{z_3}+K_{\delta_3}\right) \cap \left(\widetilde{D_1} \setminus \widetilde{C_1}\right)$ through $\widetilde{C_2}$ within this strip, we set $Z_3 = \left(\widetilde{z_3}+K_{\delta_3}\right) \cap \left(\widetilde{D_2} \setminus \widetilde{C_2}\right)$ and $Z_4 = \left(\widetilde{z_3}+K_{\delta_3}\right) \cap \left(\widetilde{D_1} \setminus \widetilde{C_1}\right)$. In this way, we obtain all the elements $V_j^s, j \in \{1,\dots,5\}$ of $\mathcal{V}^s$ which are contained in the first strip. A trajectory of $\widetilde{F}$ can leave each of these sets by moving to a set with a greater number or by moving to the next strip.
    Moreover, a trajectory of $\widetilde{F}$ can leave the set $V_5^s$ only by moving to the next strip because of the choice of the cones' heights.}
\end{figure}

We proceed in the same way in the next strips $\{z=(x,y,t) : t \in [t_i,t_{i+1})\}$, $i \in \{1,2,\dots,m(s)-1\}$ by setting the first element of the strip to be $$\left\{z=(x,y,t) \in \widetilde{D} \setminus B : t \in [t_i,t_{i+1})\right\} \setminus \left(\bigcup_{j=1}^{j(s)} Z_{j}\right)$$ and the next elements are constructed recursively as in the first strip but replacing $t<t_1$ with $t \in [t_i,t_{i+1})$. We set the last element of the relatively open partition $\mathcal{V}^s$ of $\widetilde{D}$ to be the set $B$.

Our goal is to construct an invariant partial $\varepsilon$-approximation $(\mathcal{U},\mathcal{G})$ of $\widetilde{F}$ such that $\cup (\mathbf{Green} \,\mathcal{U}) = \widetilde{D} \setminus B$ and which does not depend on the choice of $\varepsilon >0$.

We set $\mathcal{U}^1$ to be equal to $\mathcal{V}^1$ and $\mathbf{Green} \,\mathcal{U} \cap \mathcal{U}^1$ to be the family of nonempty elements of $\mathcal{V}^1$ which are contained in some ``ice-cream cone'' or ``split ice-cream cone''.
Assume that we have constructed $\mathcal{U}^s$ and $\mathbf{Green} \,\mathcal{U} \cap \mathcal{U}^s$ for some positive integer $s$. Let $U \in \mathcal{U}^s$ be an arbitrary element. If $U \in \mathbf{Green} \,\mathcal{U} \cap \mathcal{U}^s$, we leave it as it is. Otherwise, if $U$ is not an element of $\mathbf{Green} \,\mathcal{U} \cap \mathcal{U}^s$, then it is contained in the first element of one of the strips from $\mathcal{V}^s$ or it is the last element of the finite partition $\mathcal{U}^s$. In the first case we have that $U \cap \mathcal{V}^{s+1}$ is a relatively open partition of $U$. If $U$ is the last element of $\mathcal{U}^s$, then $U=B$ and we leave it as it is. Proceeding in the same way with all elements of $\mathcal{U}^s$ and using the lexicographic order, we obtain a relatively open partition $\mathcal{U}^{s+1}$ of $\widetilde{D}$ (cf. Proposition 1.5 from \cite{NR}). Let $\mathbf{Green} \,\mathcal{U} \cap \mathcal{U}^{s+1}$ be the family of nonempty elements of $\mathcal{U}^{s+1}$ which are contained in some ``ice-cream cone'' or ``split ice-cream cone''. In this way, we obtain the $\sigma$-relatively open partition $\mathcal{U} = \bigcup_{s=1}^{\infty} \mathcal{U}^s$ of $\widetilde{D}$ and its subfamily $\mathbf{Green} \,\mathcal{U}$.

All the ``ice-cream cones'' and ``split ice-cream cones'' in the construction above are contained in $\widetilde{D} \setminus B$ because each of them intersects at most one of the surfaces $\widetilde{C_1}$ and $\widetilde{C_2}$, so we have $\cup (\mathbf{Green} \,\mathcal{U}) \subseteq \widetilde{D} \setminus B$. For the reverse inclusion, suppose that $z \in \widetilde{D} \setminus B$ is arbitrary. Let $s$ be the first integer $s$ such that $z \in M_s$. Then $z$ belongs to an element of the relatively open partition $\mathcal{V}^s$ which is contained in some ``ice-cream cone'' or ``split ice-cream cone''. Thus, $z \in \cup (\mathbf{Green} \,\mathcal{U} \cap \mathcal{U}^{s}) \subseteq \cup (\mathbf{Green} \,\mathcal{U})$. Therefore, $\cup (\mathbf{Green} \,\mathcal{U}) = \widetilde{D} \setminus B$.

Let $z=(x,y,t) \in \mathbf{Green} \,\mathcal{U}$ be an arbitrary point. We want to prove that there exists a neighborhood of $z$ that intersects at most finitely many elements of $\mathbf{Green} \,\mathcal{U}$. Due to the definition of the sets $M_s$ and $\cup(\mathbf{Green} \,\mathcal{U}) = \widetilde{D} \setminus B = \cup M_s$, we can find $s$ such that $z$ belongs to the relative interior of $M_s$ with respect to $\widetilde{D}$ (for example, we can choose every positive integer $s$ such that $s > 1/\|(x,y)\|$ and $s > 1/(r-\|(x,y)\|)$). Then there exists a relative neighborhood of $z$ (with respect to $\widetilde{D}$) which is contained in $M_s$, so it is also contained in the finite union of the sets from the family $\mathbf{Green} \,\mathcal{U} \cap \mathcal{U}^s$. Thus, this neighborhood of $z$ intersects finitely many elements of $\mathbf{Green} \,\mathcal{U}$, so Definition \ref{invpapr}(ii) is satisfied.

Let $W$ be an arbitrary element of $\mathbf{Green} \,\mathcal{U}$. Then it is contained in some ``ice-cream cone'' or ``split ice-cream cone''. If $W \subset \widetilde{D_1} \setminus \widetilde{C_1}$, we define the value of $G_W$ at each point $(x,y,t) \in \overline{W}$ as $\{(y,1,1)\}$. If $W \subset \widetilde{D_2} \setminus \widetilde{C_2}$, we define the value of $G_W$ at each point $(x,y,t) \in \overline{W}$ as $\{(y,-1,1)\}$. Clearly, the multi-valued mapping $G_W : \overline{W} \rightrightarrows \mathbb{R}^3$ is upper semi-continuous with convex compact values and $G_W(x,y,t) \subseteq \widetilde{F}(x,y,t)$ for all $(x,y,t) \in \overline{W}$. It remains only to check the tangency condition in Definition \ref{invpapr}(iv.b). Let $Z$ be the ``ice-cream cone'' or the ``split ice-cream cone'' in which $W$ is contained and let $S=\left\{(x,y,t) : \overline{t} \le t < \overline{\overline{t}}\right\}$ be the strip in which $W$ is contained. We have that $Z$ is weakly forward invariant with respect to $G_W$ as an intersection of a strongly forward invariant set with respect to $G_W$ (the whole ``ice-cream cone'') and a weakly forward invariant set with respect to $G_W$ (the corresponding set $\widetilde{D_i} \setminus \widetilde{C_i}$). The strip $S$ is strongly forward invariant with respect to $G_W$. Thus, $Z \cap S$ is weakly forward invariant with respect to $G_W$, so $G_W(z) \cap T^B_{Z \cap S}(z) \ne \emptyset$ for all $z \in Z \cap S$ (cf. Theorem 2.10 from \cite{CLSW}). Since $W$ is relatively open in $Z \cap S$, we have $G_W(z) \cap T^B_{W}(z) \ne \emptyset$ for all $z \in W$.

We put $\mathcal{G}=\{G_W : W \in \mathbf{Green} \,\mathcal{U}\}$. We have constructed an invariant partial $\varepsilon$-approximation $(\mathcal{U},\mathcal{G})$ of $\widetilde{F}$ which does not depend on the choice of $\varepsilon>0$.

Now, we have to find a quasi-Lyapunov function for the so-constructed invariant partial $\varepsilon$-approximation $(\mathcal{U},\mathcal{G})$ in order to apply Theorem \ref{sol}. First, we will construct such a function on the set $\widetilde{D_1} \setminus \widetilde{C_1}$. Using axial symmetry with respect to the $t$-axis, we will extend the function to the whole set $\widetilde{D}$. We will obtain the function by constructing its level curves which uniquely determine it.

\begin{figure}[htb]
    \centering
    \includegraphics[scale=0.22]{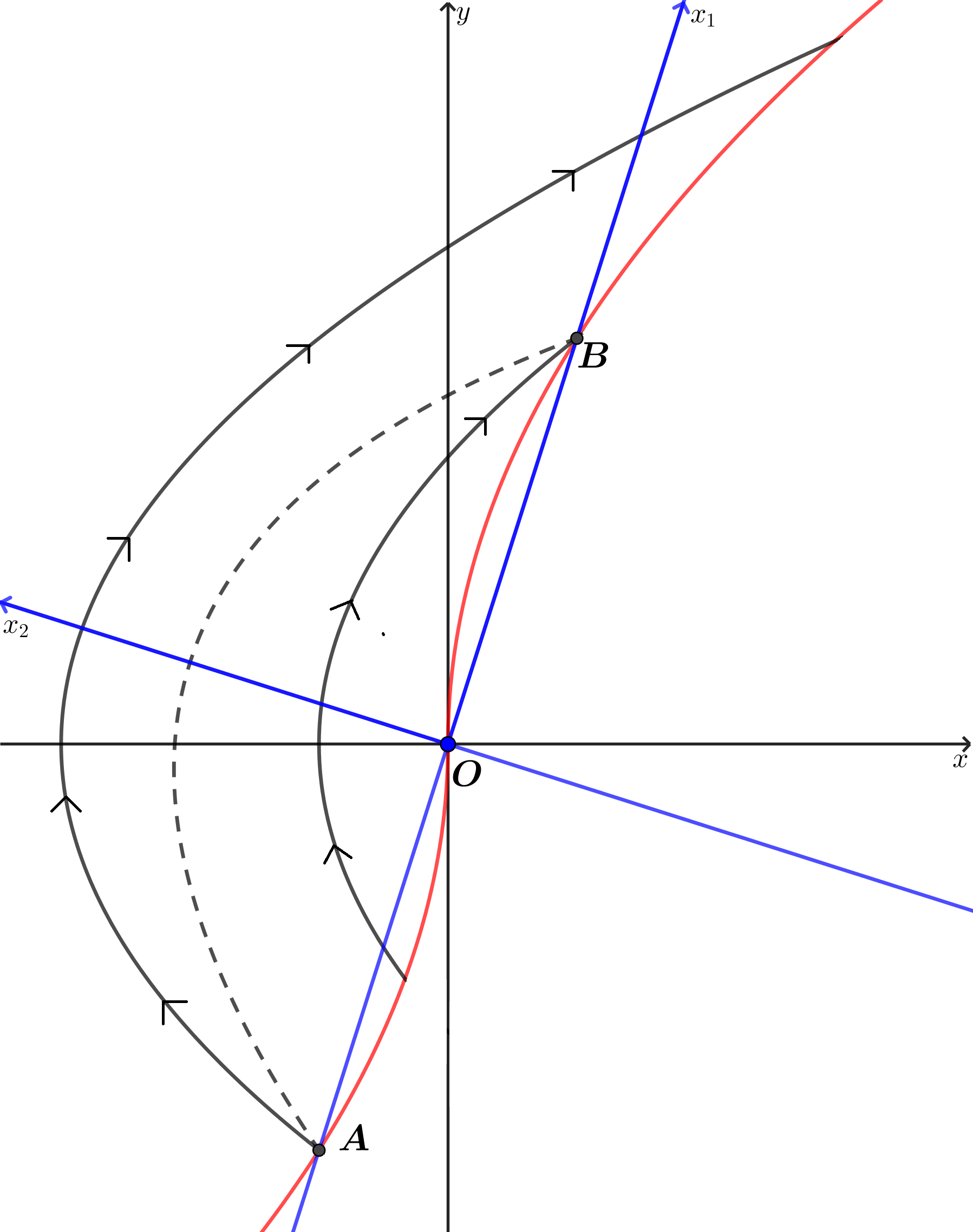}
    \caption{The construction of the parabola $\pi$ (the dotted curve). The black curves are trajectories of the inclusion.}
\end{figure}

Let $V$ be an arbitrary open ball centered at the origin and $\overline{a} > 0$ be an arbitrary positive real number. Later we will shrink the radius of $V$ and $\overline{a}$ if necessary. Let $a \in [-\overline{a},0)$. We will construct a parabola $\pi$ in the $Oxy$ plane which passes through $A\left(-\frac{1}{4}a^2,a\right)$ and $B\left(\frac{1}{4}a^2,-a\right)$. The parabolic surface $L=\left(\pi \times [0,T]\right) \cap \left(\widetilde{D_1} \setminus \widetilde{C_1}\right)$ will be contained in the level set $\left\{(x,y,t) : \overline{w}(x,y,t)=-a\right\}$ at level $-a$ of the quasi-Lyapunov function $\overline{w}$ that we will construct.

In order to obtain the equation of the parabola more easily, we transform the coordinate system by rotation about the origin and scaling. The transformed coordinate system is determined by the line passing through the points $A$ and $B$, which has an equation $y=-\frac{4x}{a}$, and by its perpendicular line passing through the origin, which has an equation $y=\frac{1}{4}ax$. More precisely, we rotate through an angle $\theta$ about the origin, where $\theta$ is the unique angle in $(0,\frac{\pi}{2})$ such that $\tan \theta = -\frac{4}{a}$. Then we scale by a factor $\frac{-a}{\sqrt{a^2+16}}$. We obtain a relation between the coordinates $(x,y)$ in the original coordinate system $Oxy$ and the coordinates $(x_1,x_2)$ in the transformed coordinate system $Ox_1x_2$ which is given by
$$\left| \begin{array}{l}
x_1 = x - \frac{4y}{a}\\  \\
x_2 = y + \frac{4x}{a}
\end{array} \right.  \mbox{  and  }
\left| \begin{array}{l}
x = \frac{a^2}{a^2+16}x_1  + \frac{4a}{a^2+16}x_2\\ \\
y = -  \frac{4a}{a^2+16}x_1 +  \frac{a^2}{a^2+16}x_2
\end{array} \right..$$

We will work in the transformed coordinate system. The points $A$ and $B$ have coordinates $\left(-\frac{1}{4}a^2-4,0\right)$ and $\left(\frac{1}{4}a^2+4,0\right)$, respectively. We are looking for a parabola with an equation $x_2 = -px_1^2 + q$, where $p>0$ and $q>0$. It has to pass through the points $A$ and $B$. Thus, $\frac{q}{p} = \left(\frac{a^2+16}{4}\right)^2$. Let's parameterize the parabola as
$$\left| \begin{array}{l}
x_1(\tau) = \tau\\
x_2(\tau) = -p \tau^2 + \frac{(a^2+16)^2}{16}p
\end{array} \right.$$
for $\tau \in I = \left[-\frac{1}{4}a^2-4,\frac{1}{4}a^2+4\right]$.
For each $\tau \in I$ the vector $(2p\tau,1)$ is normal to the parabola at the point $(x_1(\tau),x_2(\tau))$. Then for each $\tau \in I$ the vector $g(\tau)=(2p\tau,1,0)$ is normal to the parabolic surface $L$ at the point $(x_1(\tau),x_2(\tau),0)$. For each point $(x_1,x_2,x_3) \in \widetilde{D_1} \setminus \widetilde{C_1}$ the velocity vector from the invariant partial $\varepsilon$-approximation $(\mathcal{U},\mathcal{G})$, written in the transformed coordinate system, is $v(x_1,x_2,x_3)=\left(\frac{a^2}{a^2+16}\left(x_2-\frac{4}{a}x_1\right)-\frac{4}{a},1+\frac{4a}{a^2+16}\left(x_2-\frac{4}{a}x_1\right),1\right)$. Since we want to construct a quasi-Lyapunov function which is constant on the parabolic surface $L$, the gradient of this function at $(x_1(\tau),x_2(\tau),0)$ will also be normal to $L$ for each $\tau \in I$. Thus, we will choose $p$ in such a way that for each $\tau \in I$ the inner product of the normal vector $g(\tau)$ to the parabolic surface $L$ at the point $(x_1(\tau),x_2(\tau),0)$ and the velocity vector $v(x_1(\tau),x_2(\tau),0)$ at the same point is positive.

We denote by $h(\tau)=\langle g(\tau), v(x_1(\tau),x_2(\tau),0) \rangle$, $\tau \in I$, this inner product. Thus,
\begin{align*}
    h(\tau) = &\frac{2a^2p\tau}{a^2+16}\left(-p\tau^2 + \frac{(a^2+16)^2}{16}p-\frac{4}{a}\tau\right)-\frac{8p\tau}{a}+
    \\&+1+\frac{4a}{a^2+16}\left(-p\tau^2+\frac{(a^2+16)^2}{16}p-\frac{4}{a}\tau\right).
\end{align*}

The function $h$ can be written as
\begin{align*}
    h(\tau) = \frac{2a}{a^2+16}\left(-p^2 a \tau^3 - 6 p \tau^2 + \frac{p^2 a (a^2+16)^2}{16}\tau +\frac{(a^2+16)^2}{8} p\right) 
\end{align*}
$$
-\left(\frac{8 p}{a} + \frac{16}{a^2+16}\right)\tau + 1.
$$

We choose $p = -\frac{2a}{a^2+16}$. Then the second summand in the above expression equals zero and $$h(\tau) = 1 - \frac{a^2}{2} + \frac{a^4 \tau}{32 + 2 a^2} + \frac{24 a^2 \tau^2}{(16 + a^2)^2} - \frac{8 a^4 \tau^3}{(16 + a^2)^3}.$$
We have $h(\tau) > \frac{1}{2}$ for all $\tau \in I$ (we shrink $\overline{a}$ if necessary in order to make $a$ sufficiently close to zero).

In this way, we obtain a parabola with an equation $$x_2 = \frac{2a}{a^2+16} x_1^2-\frac{a(a^2+16)}{8}.$$

We denote by $f(x,y,a)$ the difference between the left- and right-hand sides of the above equation written in the original coordinate system $Oxy$, i.e. $$f(x,y,a) = y+\frac{4x}{a} - \frac{2a}{a^2+16} \left(x-\frac{4y}{a}\right)^2+\frac{a(a^2+16)}{8}.$$
Thus, $f(x,y,a)=0$ on the parabolas of interest.

Let $$M = \left\{(x,y) \in V : x \le \frac{1}{2}y^2\right\} \mbox{ and } P=M \times \left[-\overline{a},0\right).$$ 
We consider the function $f$ on the set $P$, which contains $\left(\left(D_1 \setminus C_1\right) \cap V\right) \times \left(-\overline{a},0\right)$ in its interior.

The following reasoning is going to be used repeatedly:   
\begin{equation}\label{star}
    \begin{array}{l}
        \mbox{Since }f(0,0,-\overline{a}) = -\frac{\overline{a}(\overline{a}^2+16)}{8} < 0 \mbox{ and }f \mbox{ is continuous at }(0,0,-\overline{a}), \mbox{ we}\\
        \mbox{shrink the radius of }  V \mbox{ so that }
        f(x,y,-\overline{a}) < 0 \mbox{  for all  } (x,y,-\overline{a}) \in P.
    \end{array}
\end{equation}    
We will use it in order to uniquely express $a$ as a function of $x$ and $y$ from the equation $f(x,y,a)=0$.

At each interior point of $P$ we have
\begin{gather*}
    f'_x (x,y,a) = \frac{16y-4ax}{a^2+16} + \frac{4}{a},\\
    f'_y (x,y,a) = \frac{16x}{a^2+16}-\frac{64y}{a(a^2+16)}+1,\\
    f'_a (x,y,a) = -\frac{32(2x^2+axy-2y^2)}{(a^2+16)^2}+\frac{2y^2-4x}{a^2}+\frac{2(x^2-y^2)}{a^2+16}+\frac{3a^2}{8}+2.
\end{gather*}

For $x,y$ and $a$ sufficiently close to zero (we shrink $\overline{a}$ and next, according to (\ref{star}), the radius of $V$  if necessary) we have:
\begin{equation}\label{estimate1}
    f'_x(x,y,a) \in \left(\frac{4}{a}-1,\frac{4}{a}+1\right)
\end{equation} because the first summand in $f'_x$ tends to zero whenever $(x,y,a) \to (0,0,0)$;
\begin{equation}\label{estimate2}
    f'_y(x,y,a) \in \left(-\frac{64y}{a(a^2+16)}+\frac{1}{2},-\frac{64y}{a(a^2+16)}+2\right)
\end{equation} because the first summand in $f'_y$ tends to zero whenever $(x,y,a) \to (0,0,0)$;
\begin{equation}\label{estimate3}
    f'_a(x,y,a) \in \left(\frac{2y^2-4x}{a^2}+1,\frac{2y^2-4x}{a^2}+3\right)
\end{equation} because the first, third and fourth summands in $f'_a$ tend to zero whenever $(x,y,a) \to (0,0,0)$.

For each point $(x,y)$ belonging to the interior $N$ of $M$ we want to solve the equation $f(x,y,a)=0$ with respect to $a$. Let $(x_0,y_0) \in N$ be an arbitrary point. Then $f(x_0,y_0,a)$ is a function of the variable $a$ defined on the interval $[-\overline{a},0)$. Since $2y_0^2-4x_0 > 0$, we have $f'_a(x_0,y_0,a) > 1$ for all $a \in \left(-\overline{a},0\right)$. Thus, $f(x_0,y_0,a)$ is a strictly increasing function with respect to $a$ on the interval $\left[-\overline{a},0\right)$. According to (\ref{star}), we have $f(x_0,y_0,-\overline{a}) < 0$. Moreover,  $\lim_{a \to 0^-} f(x_0,y_0,a) =$ $$  =\lim_{a \to 0^-} -\frac{4}{a} . \frac{8(y_0^2-2x_0)-a^2 x_0}{a^2+16} + y_0 + \frac{16x_0y_0-2ax_0^2}{a^2+16} + \frac{a(a^2+16)}{8} = +\infty.$$

Therefore, there exists a unique $a_0 \in (-\overline{a},0)$ such that $f(x_0,y_0,a_0)=0$. We define $\varphi(x_0,y_0):=a_0$. 
The so defined function $\varphi:N \longrightarrow \left(-\overline{a},0\right)$ is the unique implicit function satisfying
$f(x,y,\varphi(x,y))=0$ for all $(x,y) \in N$. 
The function $\varphi$ is smooth on its domain $N$ because of the smoothness of $f$ on the interior of $P$ and the implicit function theorem.

We can continuously extend the function $\varphi$ to the point $(0,0)$ with value $0$. Indeed, let $\{(x_n,y_n)\}_{n=1}^{\infty} \subset N$ be an arbitrary sequence which tends to $(0,0)$. For each point $(x_n,y_n)$ there exists a unique point $A_n=(-\frac{1}{4}a_n^2,a_n) \in C_2$, where $a_n = \varphi(x_n,y_n)$. We will prove that the corresponding sequence of points $\{A_n\}_{n=1}^{\infty}$ on $C_2$ tends to $(0,0)$. Therefore, $a_n = \varphi(x_n,y_n)$, which is the ordinate of the point $A_n$, will tend to zero whenever $n \to \infty$.

For each positive integer $n$ for a moment we will work in the transformed coordinate system $Ox_1x_2$ for $a=a_n$.
Let $Q_n\left(0,-\frac{a_n(a_n^2+16)}{8}\right)$ be the vertex of the constructed parabola for $a=a_n$. If $a_n$ is sufficiently close to zero (we shrink $\overline{a}$ and the radius of $V$ as explained in $(\ref{star})$ if necessary), then the closest point to the origin on the parabola is its vertex $Q_n$.

We return to the original coordinate system $Oxy$. We have $$\|(x_n,y_n)\| \ge \|OQ_n\| = \frac{-a_n}{\sqrt{a_n^2+16}}.\frac{-a_n(a_n^2+16)}{8}=\frac{a_n^2\sqrt{a_n^2+16}}{8}.$$ Thus, $$\frac{\|OA_n\|^2}{\|(x_n,y_n)\|} \le \frac{\|OA_n\|^2}{\|OQ_n\|} = \frac{\sqrt{a_n^2+16}}{2}.$$ Then $\|OA_n\| \le \frac{\sqrt[4]{a_n^2+16}}{\sqrt{2}} \sqrt{\|(x_n,y_n)\|}$, so $A_n$ tends to $(0,0)$ whenever $n \to \infty$. Therefore, we extend $\varphi$ to the point $(0,0)$ by $\varphi(0,0):=0$.

We obtain a function $\psi : \left(\widetilde{D_1} \setminus \widetilde{C_1}\right) \cup B \to \mathbb{R}$ defined by $\psi(x,y,t)=-\varphi(x,y)$. This function is smooth on $\left(\widetilde{D_1} \setminus \widetilde{C_1}\right)$ and continuous on $\left(\widetilde{D_1} \setminus \widetilde{C_1}\right) \cup B$ because $\varphi$ is smooth on $N$ and continuous on $N \cup \{(0,0)\}$.

Now, we will check that the so-constructed function satisfies Definition \ref{quasilyapunov}(iv) for the elements of $\mathbf{Green} \,\mathcal{U}$ which are contained in $\widetilde{D_1} \setminus \widetilde{C_1}$. The implicit function theorem implies that for all $(x,y,t) \in \left(\widetilde{D_1} \setminus \widetilde{C_1}\right)$ we have $$\psi'_x(x,y,t) = \frac{f'_x(x,y,\varphi(x,y))}{f'_a(x,y,\varphi(x,y))} \text{ and } \psi'_y(x,y,t) = \frac{f'_y(x,y,\varphi(x,y))}{f'_a(x,y,\varphi(x,y))}.$$

In order to prove the inequality from Definition \ref{quasilyapunov}(iv), we need to find an upper bound for the denominators. From $(\ref{estimate3})$ we know that $f'_a(x,y,a) < \frac{2y^2-4x}{a^2}+3$ on $\mathrm{int}\,P$. For a fixed $a \in [-\overline{a},0)$ from the construction of the parabola in the corresponding transformed coordinate system $Ox_1x_2$, it follows that the abscissas of all the points of the parabola lying in $(  {D_1} \setminus  {C_1} )\cap V$ satisfy $-\frac{1}{4}a^2-4 \le x_1 \le \frac{1}{4}a^2+4$ (remember that the coordinates of $A$ and $B$ in the transformed coordinate system are $(-\frac{1}{4}a^2-4, 0)$ and $( \frac{1}{4}a^2+4, 0)$, respectively) and the ordinates of all the points of the parabola lying in $(  {D_1} \setminus  {C_1} )\cap V$ satisfy $0 \le x_2 \le -\frac{a(a^2+16)}{8}$. Written in the original coordinate system, we have
\begin{alignat*}{2}
    -\frac{1}{4}a^2-4 &\le x-\frac{4}{a}y \le \frac{1}{4}a^2+4\\
    0 &\le \frac{4}{a}x+y \le -\frac{a(a^2+16)}{8}.
\end{alignat*}
Maximizing the upper bound for $f'_a$ is reduced to maximizing $q(x,y)=y^2-2x$ on the compact set defined by these inequalities. It follows that $q(x,y) \le \frac{a^2(a^4-16a^2+160)}{64} \le 3a^2$   (we shrink $\overline{a}$ and the radius of $V$ as explained in $(\ref{star})$ if necessary). Therefore, $f'_a(x,y,\varphi(x,y)) \le 9$ for all $(x,y) \in (  {D_1} \setminus  {C_1} )\cap V$.

Then for $(x,y,t) \in \widetilde{D_1} \setminus \widetilde{C_1}$, $y \ge 0$, using (\ref{estimate1}) and (\ref{estimate2}), we have
\begin{equation}\label{grad1}
    \begin{aligned}
        &\langle \mathbf{grad} \,\psi(x,y,t), (y,1,1) \rangle = \frac{f'_x(x,y,\varphi(x,y))}{f'_a(x,y,\varphi(x,y))} y + \frac{f'_ y(x,y,\varphi(x,y))}{f'_a(x,y,\varphi(x,y))} \\&\ge \frac{\left(\frac{4}{\varphi(x,y)}-1\right)y - \frac{64y}{\varphi(x,y)(\varphi(x,y)^2+16)} + \frac{1}{2}}{9}= \frac{\frac{4\varphi(x,y)y - (\varphi(x,y)^2+16)y}{(\varphi(x,y)^2+16)}+\frac{1}{2}}{9} > \frac{1}{36}
    \end{aligned}
\end{equation}
because the first summand in the numerator tends to $0$ whenever $(x,y) \to (0,0)$  (we shrink the radius of $V$ if necessary).

Analogously for $(x,y,t) \in \widetilde{D_1} \setminus \widetilde{C_1}$, $y<0$, using (\ref{estimate1}) and (\ref{estimate2}), we have
\begin{equation}\label{grad2}
    \begin{aligned}
        &\langle \mathbf{grad} \,\psi(x,y,t), (y,1,1) \rangle = \frac{f'_x(x,y,\varphi(x,y))}{f'_a(x,y,\varphi(x,y))} y + \frac{f'_ y(x,y,\varphi(x,y))}{f'_a(x,y,\varphi(x,y))} \\&\ge \frac{\left(\frac{4}{\varphi(x,y)}+1\right)y - \frac{64y}{\varphi(x,y)(\varphi(x,y)^2+16)} + \frac{1}{2}}{9}= \frac{\frac{4\varphi(x,y)y + (\varphi(x,y)^2+16)y}{(\varphi(x,y)^2+16)}+\frac{1}{2}}{9} > \frac{1}{36}
    \end{aligned}
\end{equation}
because the first summand in the numerator tends to $0$ whenever $(x,y) \to (0,0)$   (we shrink the radius of $V$ if necessary).

Using axial symmetry with respect to the $t$-axis, we extend the function $\psi$ to the whole set $\widetilde{D}$, and we obtain a function $w : \widetilde{D} \to \mathbb{R}$ defined by
\begin{equation*}
    w(x,y,t) = \begin{cases}
        \psi(x,y,t), &\text{if } (x,y,t) \in \widetilde{D_1} \setminus \widetilde{C_1}\\
        \psi(-x,-y,t), &\text{if } (x,y,t) \in \widetilde{D_2} \setminus \widetilde{C_2}\\
    \end{cases}.
\end{equation*}

Then we have
\begin{equation}\label{grad3}
    \begin{aligned}
        \langle \mathbf{grad} \,w(x,y,t), (y,-1,1) \rangle
    &= \langle \mathbf{grad} \,\psi(-x,-y,t), (y,-1,1) \rangle = \\
    &= \psi'_x(-x,-y,t)(-y)+\psi'_y(-x,-y,t) > \frac{1}{36}
    \end{aligned}
\end{equation}
for each point $(x,y,t) \in \widetilde{D_2} \setminus \widetilde{C_2}$.

We define the function $\overline{w}: \widetilde{D} \to \mathbb{R}$, $\overline{w}(x,y,t):=36w(x,y,t)$. It is continuous and its restriction to an arbitrary element of $\mathbf{Green} \,\mathcal{U}$ is smooth. Moreover, according to (\ref{grad1}), (\ref{grad2}) and (\ref{grad3}), we have that $\overline{w}$ satisfies Definition \ref{quasilyapunov}(iv). Hence, $\overline{w}$ is a quasi-Lyapunov function for the invariant partial $\varepsilon$-approximation $(\mathcal{U},\mathcal{G})$. Applying Theorem \ref{sol}, we prove the existence of a solution of the differential inclusion $(\ref{diffincl})$, starting from an arbitrary point of the set $V$, including the origin.

\section{Concluding remarks}

In this paper, we propose a sufficient condition for existence of a solution of a differential inclusion with a uniformly bounded right-hand side that has nonempty closed (possibly nonconvex) values (see Theorem \ref{sol} for the main result). In fact, the invariant approximations considered in the present paper (see Definition \ref{invapr} and Definition \ref{invpapr}) can be viewed as a natural extension to the patchy vector fields proposed by Ancona and Bressan in \cite{AB}. The difference is that the elements of the relatively open partitions in the current paper are weakly forward invariant, in contrast to the strong invariance hypothesis imposed in \cite{AB}. Another difference is that we do not impose any smoothness assumptions regarding the boundary of these elements and regarding the vector fields. Also, the construction in \cite{AB} allows for the trajectories to have at most finitely many points of nonsmoothness on any compact interval (see Proposition 3.1 from \cite{AB}), while the approach considered here does not impose such restrictions (for example, the trajectory starting from the origin that we obtain in the example in Section \ref{example} has countably many points of nonsmoothness).

Two motivating examples for the approach presented in the current paper are the one at the beginning of Section 3 in \cite{KR} and the one in Section \ref{example} here. The vector fields in these examples ``run away'' from the origin, but the lack of uniqueness of the solution starting from the origin poses some problems. With the notion of a quasi-Lyapunov function, we propose a way to formalize this behavior (the distance from the origin $(x,y) \mapsto \sqrt{x^2+y^2}$ can play the role of a quasi-Lyapunov function in the example from \cite{KR}).
A natural question is in what situations we can observe such behavior of the vector field. Examples of this type naturally arise from optimal control problems with chattering controls, i.e. the corresponding trajectories have infinitely many switching points on a finite time interval (cf., e.g., \cite{ZB}). 

Next, we propose two open problems. We believe that the ideas developed in this paper might help in their study. 
 
1. Prove the existence of a solution of a differential inclusion whose right-hand side is a monotone usco mapping $F$  (i.e. an upper semi-continuous monotone multi-valued mapping $F$ which has compact values).  Note that the existence of a solution is already known if we assume that the usco monotone mapping $F$ is cyclically monotone (cf.  Bressan, Cellina and Colombo \cite{BCC}). Also, the existence of an $\varepsilon$-solution in the general problem (without assuming  cyclical monotonicity)  is already obtained in \cite{KRT} for each $\varepsilon>0$.   However, the existence of a solution to the general problem is still an open question. 

2. Find a necessary and sufficient condition for existence of a solution of a differential inclusion whose right-hand side is definable in some o-minimal structure (for example, subanalytic sets). The o-minimal structures seem to be a proper setting for this problem as they possess some nice stratification properties (see \cite{Coste} for more information). This task becomes nontrivial in the presence of chattering.

\end{document}